\documentclass[11pt]{amsart}
\usepackage{amsmath,amsfonts,amssymb}
\usepackage{stmaryrd}
\usepackage[colorlinks=true, linkcolor=blue,
  citecolor=blue, urlcolor=blue]{hyperref}
\usepackage{mathtools,color}
\usepackage{pdfsync}
\def\newthm#1#2{\newtheorem{#1}[dummy]{#2}%
  \expandafter\def\csname#2\endcsname##1{\hyperref[#1:##1]{#2~\ref*{#1:##1}}}}
\newthm{thm}{Theorem}
\newthm{lemma}{Lemma}
\newthm{prop}{Proposition}
\newthm{cor}{Corollary}
\newthm{conj}{Conjecture}
\newthm{cons}{Consequence}
\newthm{fact}{Fact}
\newthm{obs}{Observation}
\newthm{guess}{Guess}
\theoremstyle{definition}
\newthm{defn}{Definition}
\newthm{remark}{Remark}
\newthm{example}{Example}
\newthm{question}{Question}
\newthm{notation}{Notation}

\newtheorem*{mainthm}{Theorem}

\newcommand{\Section}[1]{\hyperref[sec:#1]{Section~\ref*{sec:#1}}}
\newcommand{\Table}[1]{\hyperref[tab:#1]{Table~\ref*{tab:#1}}}
\newcommand{\eqn}[1]{\hyperref[eqn:#1]{(\ref*{eqn:#1})}}

\DeclareMathOperator{\QH}{QH}
\DeclareMathOperator{\QK}{QK}

\DeclareMathOperator{\pt}{pt}

\DeclareMathOperator{\codim}{codim}

\DeclareMathOperator{\dist}{dist}
\def\poly{{\mathrm{poly}}}

\DeclareMathOperator{\ch}{\ch}

\usepackage{bm}

\newcommand{\ssm}{\smallsetminus}

\newcommand{\Z}{{\mathbb Z}}

\newcommand{\cO}{{\mathcal O}}

\newcommand{\gw}[2]{\langle #1 \rangle^{\mbox{}}_{#2}}

\newcommand{\euler}[1]{\chi_{_{#1}}}

\newcommand{\al}{{\alpha}}
\newcommand{\be}{{\beta}}
\newcommand{\ga}{{\gamma}}

\DeclareMathOperator{\ev}{ev}
\newcommand{\wt}{\widetilde}
\newcommand{\wh}{\widehat}
\newcommand{\wb}{\overline}
\newcommand{\ov}{\overline}

\newcommand{\ignore}[1]{}

\newcommand{\Mb}{\wb{\mathcal M}}

\begin{document}

\title[Euler characteristics in quantum $K$-theory]{Euler
  characteristics in the quantum $K$-theory of flag varieties}

\date{March 6, 2019}

\author[A.~S.~Buch]{Anders~S.~Buch}
\address{Department of Mathematics, Rutgers University, 110
  Frelinghuysen Road, Piscataway, NJ 08854, USA}
\email{asbuch@math.rutgers.edu}

\author[S.~Chung]{Sjuvon~Chung}
\address{Department of Mathematics, The Ohio State University, 100 Math Tower, 231 W 18th Avenue, Columbus, OH 43210, USA}
\email{chung.809@osu.edu}

\author[C.~Li]{Changzheng~Li}
\address{School of Mathematics, Sun Yat-sen University,
  Guangzhou 510275, P.R. China}
\email{lichangzh@mail.sysu.edu.cn}

\author[L.~C.~Mihalcea]{Leonardo~C.~Mihalcea}
\address{460 McBryde Hall, Department of Mathematics, Virginia Tech,
  Blacksburg, VA 24061, USA}
\email{lmihalce@math.vt.edu}


\thanks{The authors acknowledge support from NSF grant DMS-1503662
  (Buch), NSFC grants 11771455, 11831017, and Guangdong Introducing Innovative 
  and Enterpreneurial Teams No. 2017ZT07X355 (Li), and the NSA Young
  Investigator Award H98320-16-1-0013 and a Simons Collaboration
  Grant (Mihalcea).}

\begin{abstract}
  We prove that the sheaf Euler characteristic of the product of a
  Schubert class and an opposite Schubert class in the quantum
  $K$-theory ring of a (generalized) flag variety $G/P$ is equal to
  $q^d$, where $d$ is the smallest degree of a rational curve joining
  the two Schubert varieties.  This implies that the sum of the
  structure constants of any product of Schubert classes is equal to
  1.  Along the way, we provide a description of the smallest degree
  $d$ in terms of its projections to flag varieties defined by maximal
  parabolic subgroups.
\end{abstract}

\maketitle


\section{Introduction}

The goal of this paper is to relate distances between Schubert
varieties in a complex flag variety $X = G/P$ to products of Schubert
classes in the quantum $K$-theory ring $\QK_T(X)$.

The torus-equivariant $K$-theory ring $K_T(X)$ is an algebra over the
ring of virtual representations $\Gamma = K_T(\pt)$ of the maximal
torus in $G$.  As a module over $\Gamma$, the ring $K_T(X)$ has a
basis consisting of the classes $\cO_v = [\cO_{X_v}]$ of the Schubert
varieties $X_v \subset X$, and another basis consisting of the classes
$\cO^u = [\cO_{X^u}]$ of the opposite Schubert varieties $X^u$.

Let $\Gamma\llbracket q \rrbracket$ be the ring of formal power series
in variables $q_\be$ that correspond to the Schubert basis
$\{[X_{s_\be}]\}$ of $H_2(X,\Z)$.  Given any degree
$d = \sum_\be d_\be [X_{s_\be}]$ in $H_2(X,\Z)$ we write
$q^d = \prod_\be q_\be^{d_\be}$.  The (small, equivariant) quantum
$K$-theory ring $\QK_T(X)$ of Givental \cite{givental:wdvv} and Lee
\cite{lee:quantum} is a $\Gamma\llbracket q \rrbracket$-algebra, which
as a module over $\Gamma\llbracket q \rrbracket$ can be defined by
$\QK_T(X) = K_T(X) \otimes_\Gamma \Gamma\llbracket q \rrbracket$.  The
product in $\QK_T(X)$ takes the form
\begin{equation}\label{eqn:qkprod}
  \cO^u \star \cO^v = \sum_{w,d} N^{w,d}_{u,v}\, q^d\, \cO^w \,,
\end{equation}
where the Schubert structure constants $N^{w,d}_{u,v} \in \Gamma$ are
defined in terms of the $K$-theoretic Gromov-Witten invariants of $X$.

The sheaf Euler characteristic map $\euler{X} : K_T(X) \to \Gamma$ is
defined by
\[
  \euler{X}([E]) = \sum_{i \geq 0} (-1)^i\, [H^i(X,E)]
\]
for any equivariant vector bundle $E$.  Equivalently, $\euler{X}$ is
the unique $\Gamma$-linear map satisfying
$\euler{X}(\cO^u) = \euler{X}(\cO_v) = 1$.  Let
$\chi : \QK_T(X) \to \Gamma\llbracket q \rrbracket$ denote the
$\Gamma\llbracket q \rrbracket$-linear extension of $\euler{X}$.  Our
main result is the following theorem.

\begin{mainthm}
  We have $\chi(\cO^u \star \cO_v) = q^{\dist_X(u,v)}$, where
  $\dist_X(u,v) \in H_2(X,\Z)$ denotes the smallest degree of a
  rational curve connecting $X^u$ to $X_v$.
\end{mainthm}

This result was proved earlier in \cite{buch.chung:euler} by Buch and
Chung when $X$ is a \emph{cominuscule flag variety}, such as a
Grassmann variety of Lie type A or a maximal isotropic Grassmannian of
type B, C, or D.

Implicit in the statement is the claim that, given opposite Schubert
varieties $X^u$ and $X_v$ in $X$, there exists a unique minimal degree
$\dist_X(u,v)$ in $H_2(X,\Z)$ for which $X^u$ and $X_v$ can be
connected by a rational curve of this degree.  We will refer to this
degree as the \emph{distance} between $X^u$ and $X_v$.  Results of
Fulton and Woodward \cite{fulton.woodward:quantum} show that this
question is equivalent to the existence of a minimal degree $d$ for
which $q^d$ occurs in the product $[X^u] \star [X_v]$ in the small
quantum cohomology ring $\QH(X)$.  Postnikov has proved in
\cite{postnikov:quantum*1} that this is true when $X = G/B$ is a
variety of complete flags.  We deduce the existence of a minimal
degree in general from these results.  We also show that
$\dist_X(u,v)$ can be expressed in terms of distances in flag
varieties defined by maximal parabolic subgroups.

Let $\Mb_{0,3}(X,d)$ denote the Kontsevich moduli space of stable maps
to $X$ of degree $d$ and genus zero.  Any 2-pointed $K$-theoretic
Gromov-Witten invariant $\gw{\cO^u,\cO_v}{d}$ can be interpreted as
the sheaf Euler characteristic of the Gromov-Witten variety
$\ev_1^{-1}(X^u) \cap \ev_2^{-1}(X_v)$ in $\Mb_{0,3}(X,d)$.  It was
proved in \cite{buch.chaput.ea:finiteness} that this Gromov-Witten
variety is either empty or unirational with rational singularities.
Since it is non-empty if and only if there exists a rational curve of
degree $d$ from $X^u$ to $X_v$, the 2-point Gromov-Witten invariants
of $X$ are determined by the formula
\[
  \gw{\cO^u, \cO_v}{d} = \begin{cases}
    1 & \text{if $d \geq \dist_X(u,v)$;}\\
    0 & \text{otherwise.}
  \end{cases}
\]
We show that this is equivalent to our identity
$\euler{X}(\cO^u \star \cO_v) = q^{\dist_X(u,v)}$ by using the
Frobenius property of the product $\star$ of $\QK_T(X)$.  Still
another equivalent formulation is that the small quantum $K$-metric of
$X$ is given by
\[
  (\!( \cO^u, \cO_v )\!) =
  \frac{q^{\dist_X(u,v)}}{\prod_\be (1 - q_\be)} \,.
\]

According to the definition of the quantum $K$-theory ring $\QK_T(X)$,
the product \eqn{qkprod} of two Schubert classes could potentially
have infinitely many non-zero terms.  Our main result directly implies
that, for all but finitely many degrees $d$, the sum of structure
constants $\sum_w N^{w,d}_{u,v}$ is equal to zero.  In addition, the
sum of these sums over all degrees $d$ is equal to 1:
\[
  \sum_d \sum_w N^{w,d}_{u,v} = 1 \,.
\]

It has recently been established that the quantum $K$-theory ring
$\QK_T(X)$ satisfies \emph{finiteness}, in the sense that only
finitely many of the structure constants $N^{w,d}_{u,v}$ are non-zero.
Equivalently, the quantum $K$-theory ring contains a subring defined
by $\QK_T^\poly(X) = K_T(X) \otimes_\Gamma \Gamma[q]$.  This was
proved by Buch, Chaput, Mihalcea, and Perrin when the Picard group of
$X$ has rank one \cite{buch.mihalcea:quantum,
  buch.chaput.ea:finiteness, buch.chaput.ea:rational}, by Kato
\cite{kato:loop, kato:frobenius} for varieties of complete flags
$G/B$, and by Anderson, Chen, and Tseng
\cite{anderson.chen.ea:quantum} for arbitrary flag varieties $G/P$.
In addition, Kato's work proves a relation between the quantum
$K$-theory ring $\QK_T(G/B)$ and the equivariant $K$-homology of
affine Grassmannians that was conjectured in
\cite{lam.li.ea:conjectural}.

The sheaf Euler characteristic map $\euler{X} : K_T(X) \to \Gamma$ is
almost never a ring homomorphism, for example because
$\euler{X}(\cO^u \cdot \cO_v) = 0 \neq 1 = \euler{X}(\cO^u) \cdot
\euler{X}(\cO_v)$ whenever $X^u$ and $X_v$ are disjoint Schubert
varieties in $X$.  However, our main theorem implies that $\euler{X}$
lifts to a well-defined ring homomorphism
$\wh\chi : \QK_T^\poly(X) \to \Gamma$ defined by
$\wh\chi(\cO^u) = \wh\chi(q_\be) = 1$.  We consider this as evidence
that the quantum $K$-theory ring is a natural construction.

In \Section{dist} we prove the existence of the minimal degree
$\dist_X(u,v)$ based on Fulton, Woodward, and Postnikov's results
about quantum cohomology.  \Section{qk} then defines the quantum
$K$-theory ring $\QK_T(X)$ and proves the identity
$\chi(\cO^u \star \cO_v) = q^{\dist_X(u,v)}$ and its consequences.

We thank David Anderson for helpful conversations.

\section{The distance between Schubert varieties}
\label{sec:dist}

\subsection{Flag varieties}

Let $X = G/P$ be a flag variety defined by a connected semisimple
complex Lie group $G$ and a parabolic subgroup $P$.  Fix a maximal
torus $T$ and a Borel subgroup $B$ such that
$T \subset B \subset P \subset G$.  The opposite Borel subgroup
$B^- \subset G$ is defined by $B \cap B^- = T$.  Let $W = N_G(T)/T$ be
the Weyl group of $G$, $W_P = N_P(T)/T$ the Weyl group of $P$, and let
$W^P \subset W$ be the subset of minimal representatives of the cosets
in $W/W_P$.  Let $\Phi$ denote the root system of $G$, with positive
roots $\Phi^+$ and simple roots $\Delta \subset \Phi^+$.  The
parabolic subgroup $P$ is determined by the subset
$\Delta_P = \{ \be \in \Delta \mid s_\be \in W_P \}$.  Each element
$w \in W$ defines a $B$-stable Schubert variety $X_w = \ov{Bw.P}$ and
a $B^-$-stable (opposite) Schubert variety $X^w = \ov{B^-w.P}$.  If
$w \in W^P$ is a minimal representative, then
$\dim(X_w) = \codim(X^w,X) = \ell(w)$.

The group $H_2(X,\Z)$ is a free $\Z$-module, with a basis consisting
of the Schubert classes $[X_{s_\be}]$ for
$\be \in \Delta\ssm\Delta_P$.  Given two elements
$d = \sum_\be d_\be [X_{s_\be}]$ and
$d' = \sum_\be d'_\be [X_{s_\be}]$ expressed in this basis, we write
$d \leq d'$ if and only if $d_\be \leq d'_\be$ for each
$\be \in \Delta\ssm\Delta_P$.  This defines a partial order on
$H_2(X,\Z)$.

For any root $\al \in \Phi$ that is not in the span $\Phi_P$ of
$\Delta_P$, there exists a unique irreducible $T$-invariant curve
$X(\al) \subset X$ that connects the points $1.P$ and $s_\al.P$.  An
arbitrary irreducible $T$-invariant curve $C \subset X$ has the form
$C = w.X(\al)$ for some $w \in W$ and $\al \in \Phi^+ \ssm \Phi_P$.

\subsection{Quantum cohomology}

Given an effective degree $d \geq 0$ in $H_2(X,\Z)$ we let
$\Mb_{0,n}(X,d)$ denote the Kontsevich moduli space of all $n$-pointed
stable maps $f : C \to X$ of arithmetic genus zero and degree
$f_*[C] = d$.  This space is equipped with evaluation maps
$\ev_i : \Mb_{0,n}(X,d) \to X$ for $1 \leq i \leq n$, where $\ev_i$
sends a stable map to the image of the $i$-th marked point in its
domain.  Given cohomology classes $\ga_1, \dots, \ga_n \in H^*(X,\Z)$,
the corresponding (cohomological) Gromov-Witten invariant of degree
$d$ is defined by
\[
  \gw{\ga_1, \dots, \ga_n}{d} = \int_{\Mb_{0,n}(X,d)} \ev_1^*(\ga_1)
  \wedge \dots \wedge \ev_n^*(\ga_n) \,.
\]

Let $\Z[q] = \Z[q_\be : \be \in \Delta\ssm\Delta_P]$ denote a
polynomial ring in variables $q_\be$ corresponding to the basis
elements of $H_2(X,\Z)$.  Given any degree
$d = \sum_\be d_\be [X_{s_\be}] \in H_2(X,\Z)$, we will write
$q^d = \prod_\be q_\be^{d_\be}$.  The (small) quantum cohomology ring
$\QH(X)$ is a $\Z[q]$-algebra which as a $\Z[q]$-module can be defined
by $\QH(X) = H^*(X,\Z) \otimes_\Z \Z[q]$.  The product is defined by
\[
  \ga_1 \star \ga_2 =
  \sum_{w, d \geq 0} \gw{\ga_1, \ga_2, [X_w]}{d}\, q^d\, [X^w]
\]
for $\ga_1, \ga_2 \in H^*(X,\Z)$.  Here we identify any class
$\ga \in H^*(X,\Z)$ with $\ga \otimes 1 \in \QH(X)$.

In the following, the image of a stable map to $X$ will be called a
stable curve in $X$.  Given a stable curve $C \subset X$, we let $[C]$
denote the degree in $H_2(X;\Z)$ defined by $C$.  In particular, we
set $[C]=0$ if $C$ is a single point.  The following theorem is a
subset of \cite[Thm.~9.1]{fulton.woodward:quantum}.

\begin{thm}[Fulton and Woodward]\label{thm:fw}
  Let $u,v \in W^P$ and $d \in H_2(X,\Z)$.  The Schubert varieties
  $X^u$ and $X_v$ can be connected by a stable curve of degree $d$ if
  and only if $q^{d'}$ occurs in the product $[X^u] \star [X_v]$ for
  some $d' \leq d$.
\end{thm}

Let $Y = G/B$ denote the variety of complete flags.  In this case the
following was proved in \cite[Cor.~3]{postnikov:quantum*1}.

\begin{thm}[Postnikov]\label{thm:postnikov}
  Let $u,v \in W$.  There is a unique minimal degree $d_{\min}(u,v)$
  in $H_2(Y,\Z)$ for which $q^{d_{\min}(u,v)}$ occurs in
  $[Y^u] \star [Y^v]$.
\end{thm}

\subsection{Distance}

In this section we extend Postnikov's result to arbitrary flag
varieties $X = G/P$.  For each simple root $\be \in \Delta$ we set
$Z_\be = G/P_\be$, where $P_\be \subset G$ is the unique maximal
parabolic subgroup containing $B$ for which $s_\be \notin W_{P_\be}$.
The group $H_2(Z_\be,\Z)$ is free of rank one, generated by
$[(Z_\be)_{s_\be}]$.  To simplify notation we identify $H_2(Z_\be,\Z)$
with $\Z$, by identifying $[(Z_\be)_{s_\be}]$ with 1.  Let
$\pi_\be : X \to Z_\be$ denote the projection.  Any degree
$d \in H_2(X,\Z)$ is then given by $d = \sum_\be d_\be [X_{s_\be}]$,
where $d_\be = (\pi_\be)_*(d)$.

Given $u, v \in W$ we let $\dist_\be(u,v) \in \Z$ denote the smallest
degree of a stable curve in $Z_\be$ connecting the Schubert varieties
$(Z_\be)^u$ and $(Z_\be)_v$.  This is well defined since the natural
numbers are well ordered.  Define the \emph{distance} between $X^u$
and $X_v$ to be the class in $H_2(X,\Z)$ given by
\begin{equation}\label{eqn:dist}
  \dist_X(u,v) =
  \sum_{\be \in \Delta\ssm\Delta_P} \dist_\be(u,v)\, [X_{s_\be}] \,.
\end{equation}
This terminology is justified by \Theorem{dist} below.  In the case
where both $X^u$ and $X_v$ are single points, the identity \eqn{dist}
can also be found in \cite{barligea:curve}.

\begin{lemma}\label{lemma:curve_beta}
  Let $u,v \in W$ and $\be \in \Delta\ssm \Delta_P$.  Then $X^u$ and
  $X_v$ can be connected by a stable curve $C \subset X$ for which
  $(\pi_\be)_*[C] = \dist_\be(u,v)$.
\end{lemma}
\begin{proof}
  Since the intersection
  $\ev_1^{-1}((Z_\be)^u) \cap \ev_2^{-1}((Z_\be)_v)$ is a closed
  subvariety of $\Mb_{0,3}(Z_\be,\dist_\be(u,v))$ that is invariant
  under the action of $T$, we may find a $T$-invariant stable curve
  $\wt C \subset Z_\be$ of degree $\dist_\be(u,v)$ that connects
  $(Z_\be)^u$ and $(Z_\be)_v$.  This implies that $\wt C$ is a chain
  of irreducible $T$-invariant curves, that is, there exist
  $\kappa_0, \kappa_1, \dots, \kappa_m \in W$ and
  $\al_1,\dots,\al_m \in \Phi^+\ssm \Phi_P$ such that
  $\kappa_0.P_\be \in (Z_\be)^u$, $\kappa_m.P_\be \in (Z_\be)_v$,
  $\kappa_i = \kappa_{i-1} s_{\al_i}$ for $1 \leq i \leq m$, and
  $\wt C = \kappa_1.Z_\be(\al_1) \cup \dots \cup
  \kappa_m.Z_\be(\al_m)$.  Since we have $\kappa_0 w' \geq u$ for some
  $w' \in W_{P_\be}$, there exists a stable curve $C' \subset X$
  connecting $\kappa_0.P$ to a point in $X^u$, such that
  $(\pi_\be)_*[C'] = 0$.  Similarly we can find $C'' \subset X$
  connecting $\kappa_m.P$ to a point in $X_v$ such that
  $(\pi_\be)_*[C''] = 0$.  We can therefore take $C \subset X$ to be
  the union of $C'$, $C''$, and the curves $\kappa_i.X(\al_i)$ for
  $1 \leq i \leq m$.
\end{proof}

Recall that $[Y_v] = [Y^{w_0 v}]$, where $w_0$ is the longest element
of $W$.  It therefore follows from \Theorem{postnikov} that
$q^{d_{\min}(u,w_0 v)}$ is the unique minimal power of $q$ that occurs
in the product $[Y^u] \star [Y_v] \in \QH(Y)$.

\begin{lemma}\label{lemma:dist=min}
  We have $\dist_Y(u,v) = d_{\min}(u,w_0 v)$.
\end{lemma}
\begin{proof}
  Write
  $d = d_{\min}(u,w_0v) = \sum_{\be \in \Delta} d_\be [Y_{s_\be}]$ and
  let $\be \in \Delta$ be given.  Since $q^d$ occurs in the quantum
  product $[Y^u] \star [Y_v]$, there exists a stable curve
  $C \subset Y$ of degree $d$ from $Y^u$ to $Y_v$ by \Theorem{fw}.
  Since $\pi_\be(C) \subset Z_\be$ is a curve of degree at most
  $d_\be$ from $(Z_\be)^u$ to $(Z_\be)_v$, we have
  $d_\be \geq \dist_\be(u,v)$.  On the other hand, according to
  \Lemma{curve_beta} we can find a stable curve $C \subset Y$ from
  $Y^u$ to $Y_v$ such that $(\pi_\be)_*[C] = \dist_\be(u,v)$.
  \Theorem{fw} then implies that the product $[Y^u] \star [Y_v]$
  contains a power $q^{d'}$ for which $d' \leq [C]$, and the
  minimality of $d_{\min}(u,w_0 v)$ implies that
  $d_{\min}(u,w_0 v) \leq d'$.  We deduce that
  $d_\be \leq \dist_\be(u,v)$, which completes the proof.
\end{proof}

\begin{thm}\label{thm:dist}
  Let $u, v \in W$ and $d \in H_2(X,\Z)$.  There exists a stable curve
  of degree $d$ from $X^u$ to $X_v$ if and only if
  $d \geq \dist_X(u,v)$.
\end{thm}
\begin{proof}
  Write $d = \sum_\be d_\be [X_{s_\be}]$.  The implication `only if'
  follows because $(\pi_\be)_*(d) = d_\be$, as in the proof of
  \Lemma{dist=min}.  Assume that $d \geq \dist_X(u,v)$ and define
  $d' = \sum_{\be \in \Delta} d'_\be [Y_{s_\be}] \in H_2(Y,\Z)$ by
  $d'_\be = d_\be$ for $\be \in \Delta\ssm\Delta_P$ and
  $d'_\be = \dist_\be(u,v)$ for $\be \in \Delta_P$.  Since
  $d' \geq \dist_Y(u,v)$, it follows from \Theorem{fw} that there
  exists a stable curve $C \subset Y$ of degree $d'$ from $Y^u$ to
  $Y_v$.  The image of $C$ under the projection $Y \to X$ is a curve
  from $X^u$ to $X_v$ of degree at most $d$.  By possibly attaching
  some extra components, we obtain the desired stable curve of degree
  $d$.
\end{proof}

\section{Quantum $K$-theory}
\label{sec:qk}

\subsection{$K$-theory}

The equivariant $K$-theory ring $K_T(X)$ is the Grothendieck ring of
$T$-equi\-vari\-ant algebraic vector bundles on $X$, equipped with the
product coming from the tensor product of vector bundles.  This ring
is an algebra over the ring $\Gamma = K_T(\pt)$ of virtual
representations of $T$, with a basis consisting of the Schubert
classes $\cO_v = [\cO_{X_v}]$ for $v \in W^P$; the opposite Schubert
classes $\cO^u = [\cO_{X^u}]$ for $u \in W^P$ form another basis.  The
ring $\Gamma$ can be identified with a ring of Laurent polynomials in
as many variables as the rank of $T$.

The sheaf Euler characteristic map $\euler{X} : K_T(X) \to \Gamma$ is
defined as the pushforward map along the structure morphism
$X \to \{\pt\}$; that is,
\[
  \euler{X}([E]) = \sum_{i \geq 0} (-1)^i\, [H^i(X,E)] \,,
\]
where the sheaf cohomology group $H^i(X,E)$ is regarded as a
representation of $T$.  Equivalently, $\euler{X}$ is the unique
$\Gamma$-linear map defined by
$\euler{X}(\cO^w) = \euler{X}(\cO_w) = 1$
\cite{ramanan.ramanathan:projective,ramanathan:schubert}.  More
generally, if $Z \subset X$ is any closed $T$-invariant subvariety
that is unirational and has rational singularities, then
$\euler{X}([\cO_Z]) = 1$ \cite[Cor.~4.18]{debarre:higher-dimensional}.

\subsection{Gromov-Witten invariants}

Given classes $\ga_1, \dots, \ga_n \in K_T(X)$ and a degree
$d \in H_2(X,\Z)$, the corresponding (equivariant, $K$-theoretic,
$n$-pointed, genus zero) Gromov-Witten invariant of $X$ is defined by
\[
  \gw{\ga_1,\dots,\ga_n}{d} = \euler{\Mb_{0,n}(X,d)}( \ev_1^*(\ga_1)
  \cdot \ldots \cdot \ev_n^*(\ga_n)) \ \in \Gamma \,.
\]
Since the moduli space $\Mb_{0,n}(X,d)$ is empty for $d=0$ and
$n \leq 2$, we will use the convention that
$\gw{\ga_1,\dots,\ga_n}{0} = \euler{X}(\ga_1 \cdot \ldots \cdot
\ga_n)$ for any $n \geq 0$.  This is consistent with the above
definition since the general fibers of the forgetful map
$\Mb_{0,n+1}(X,d) \to \Mb_{0,n}(X,d)$ are projective lines and all
evaluation maps on $\Mb_{0,n}(x,0)$ are identical (see
\cite[Thm.~7.1]{kollar:higher} or
\cite[Thm.~3.1]{buch.mihalcea:quantum}).

The two-point Gromov-Witten invariant $\gw{\cO^u, \cO_v}{d}$ is equal
to the sheaf Euler characteristic of the Gromov-Witten variety
$\ev_1^{-1}(X^u) \cap \ev_2^{-1}(X_v)$, which by
\cite[Cor.~3.3]{buch.chaput.ea:finiteness} is either empty or
unirational with rational singularities.  Since the question of
non-emptiness is determined by \Theorem{dist}, we obtain the following
identity, interpreting the formula for the 2-point K-theoretic Gromov-Witten 
invariants from \cite[Remark 7.5]{buch.mihalcea:nbhds} in terms of the distance function.

\begin{prop}\label{prop:gw2pt}
  For $u,v \in W^P$ and $d \in H_2(X,\Z)$ we have
  \[
    \gw{\cO^u, \cO_v}{d} = \begin{cases}
      1 & \text{if $d \geq \dist_X(u,v)$;}\\
      0 & \text{otherwise.}
    \end{cases}
  \]
\end{prop}

\Proposition{gw2pt} for $d=0$ recovers the well known fact that
$\euler{X}(\cO^u \cdot \cO_v) = 1$ whenever $u \leq v$ in the Bruhat
order on $W^P$, and $\euler{X}(\cO^u \cdot \cO_v) = 0$ otherwise.  In
particular, the bilinear map $K_T(X) \times K_T(X) \to \Gamma$ defined
by $(\ga_1,\ga_2) \mapsto \euler{X}(\ga_1 \cdot \ga_2)$ is a perfect
pairing.

\subsection{Quantum $K$-theory}

The (small, equivariant) quantum $K$-theory ring $\QK_T(X)$ of
Givental \cite{givental:wdvv} and Lee \cite{lee:quantum} is an algebra
over the ring of formal power series
$\Gamma\llbracket q\rrbracket = \Gamma\llbracket q_\be : \be \in
\Delta\ssm\Delta_P\rrbracket$.  As a module over
$\Gamma\llbracket q\rrbracket$ it is defined by
$\QK_T(X) = K_T(X) \otimes_\Gamma \Gamma\llbracket q\rrbracket$.  The
\emph{quantum $K$-metric} is the
$\Gamma\llbracket q\rrbracket$-bilinear pairing
$\QK_T(X) \times \QK_T(X) \to \Gamma\llbracket q \rrbracket$
determined by
\[
  (\!( \ga_1, \ga_2 )\!) \ = \ \sum_{d \geq 0} q^d \gw{\ga_1, \ga_2}{d}
\]
for $\ga_1, \ga_2 \in K_T(X)$.  The \emph{quantum product} $\star$ is
the unique $\Gamma\llbracket q \rrbracket$-bilinear product
$\QK_T(X) \times \QK_T(X) \to \QK_T(X)$ determined by
\[
  (\!( \ga_1 \star \ga_2, \ga_3 )\!) \ = \
  \sum_{d \geq 0} q^d \gw{\ga_1, \ga_2, \ga_3}{d}
\]
for all $\ga_1, \ga_2, \ga_3 \in K_T(X)$.  It was proved by Givental
\cite{givental:wdvv} that this product $\star$ is associative.  The
symmetry of Gromov-Witten invariants implies that the Frobenius
property
$(\!( \ga_1 \star \ga_2, \ga_3 )\!) = (\!( \ga_1, \ga_2 \star \ga_3
)\!)$ holds for all $\ga_1, \ga_2, \ga_3 \in \QK_T(X)$.  The string
identity $\gw{\ga_1,\dots,\ga_n,1}{d} = \gw{\ga_1,\dots,\ga_n}{d}$
implies that $1 \in K_T(X)$ is a multiplicative unit in $\QK_T(X)$.

\begin{remark}
  Consider the formal linear combination
  $t = \sum_{u \in W^P} t_u \cO^u$, where the coefficients $t_u$ are
  independent commuting variables.  The \emph{quantum $K$-potential}
  of $X$ is the generating function
  \[
    G(t,q) = \sum_{n \geq 0} \sum_{d \geq 0} \frac{q^d}{n!}
    \gw{t,\dots,t}{d,n} \,,
  \]
  where $\gw{t,\dots,t}{d,n} = \gw{t,\dots,t}{d}$ denotes an
  $n$-pointed Gromov-Witten invariant with $t$ repeated $n$ times.
  The quantum $K$-metric can be obtained from $G(t,q)$ as
  \[
    (\!( \cO^u, \cO^v )\!) \ = \ \frac{\partial}{\partial t_u}
    \frac{\partial}{\partial t_v} G(t,u) \bigg\vert_{t=0} \,,
  \]
  and the quantum product $\star$ is determined by
  \[
    (\!( \cO^u \star \cO^v, \cO^w )\!) \ = \
    \frac{\partial}{\partial t_u} \frac{\partial}{\partial t_v}
    \frac{\partial}{\partial t_w} G(t,u) \bigg\vert_{t=0} \,.
  \]
  If we do not specialize the variables $t_u$ to zero, then we arrive
  at the big quantum $K$-theory ring of $X$.
\end{remark}

\subsection{Sums of structure constants}

Let $\chi : \QK_T(X) \to \Gamma\llbracket q \rrbracket$ denote the
$\Gamma\llbracket q\rrbracket$-linear extension of the sheaf Euler
characteristic map $\euler{X}$.  The following identity is our main
result.  It was known in the special case where $X$ is a cominuscule
flag variety \cite{buch.chung:euler}.

\begin{thm}\label{thm:euler=dist}
  For $u,v \in W^P$ we have
  $\chi(\cO^u \star \cO_v) = q^{\dist_X(u,v)}$.
\end{thm}
\begin{proof}
  It follows from \Proposition{gw2pt} that
  \[
    (\!( \cO^u, \cO_v )\!) =
    \frac{q^{\dist_X(u,v)}}{\prod_\be (1 - q_\be)} \,,
  \]
  where the product is over $\be \in \Delta\ssm\Delta_P$.  This
  identity implies that
  \[
    \chi(\ga) = (\!( \ga, 1 )\!)\, {\textstyle\prod_\be (1 - q_\be)}
  \]
  for any class $\ga \in \QK_T(X)$.  We obtain
  \[\begin{split}
      \chi(\cO^u \star \cO_v) \
      &= \ (\!( \cO^u \star \cO_v, 1 )\!)\, {\textstyle\prod_\be (1 - q_\be)}
      \ = \ (\!( \cO^u, \cO_v \star 1 )\!)\, {\textstyle\prod_\be (1 - q_\be)} \\
      &= \ (\!( \cO^u, \cO_v )\!)\, {\textstyle\prod_\be (1 - q_\be)}
      \ = \ q^{\dist_X(u,v)} \,,
    \end{split}\]
  as required.
\end{proof}

The Schubert structure constants of the quantum $K$-theory ring
$\QK_T(X)$ are the classes $N^{w,d}_{u,v} \in \Gamma$, indexed by
$u,v,w \in W^P$ and $d \in H_2(X,\Z)$, defined by
\[
  \cO^u \star \cO^v = \sum_{w, d \geq 0} N^{w,d}_{u,v}\, q^d\, \cO^w \,.
\]

\begin{cor}\label{cor:sumcoef}
  Let $u, v \in W^P$.  For all but finitely many degrees
  $d \in H_2(X,\Z)$, the sum $\sum_{w \in W^P} N^{w,d}_{u,v}$ is equal
  to zero.  Moreover, we have
  \[
    \sum_{d \in H_2(X,\Z)} \sum_{w \in W^P} N^{w,d}_{u,v} \ = \ 1 \,.
  \]
\end{cor}
\begin{proof}
  Write $\cO^v = \sum_{z \in W^P} f_z\, \cO_z$ where $f_z \in \Gamma$.
  We then have
  \[
    \sum_{w, d \geq 0} N^{w,d}_{u,v}\, q^d =
    \chi(\cO^u \star \cO^v) =
    \sum_z f_z\, \chi(\cO^u \star \cO_z) =
    \sum_z f_z\, q^{\dist_X(u,z)} \,.
  \]
  The first claim follows because the last sum has finitely many
  terms, and the second claim holds because
  $\sum_z f_z = \euler{X}(\cO^v) = 1$.
\end{proof}

\subsection{Ring homomorphism}

Let $\Gamma[q] \subset \Gamma\llbracket q \rrbracket$ be the subring
of polynomials in the variables $q_\be$, and set
$\QK_T^\poly(X) = K_T(X) \otimes_\Gamma \Gamma[q]$.  It has recently
been proved that the quantum $K$-theory ring $\QK_T(X)$ satisfies
\emph{finiteness}, that is, this ring contains $\QK_T^\poly(X)$ as a
subring \cite{buch.mihalcea:quantum, buch.chaput.ea:finiteness,
  buch.chaput.ea:rational, kato:loop, kato:frobenius,
  anderson.chen.ea:quantum}.  While the sheaf Euler characteristic map
$\euler{X} : K_T(X) \to \Gamma$ is a ring homomorphism only when $X$
is a single point, our last result shows that this changes if $K_T(X)$
is replaced by $\QK_T^\poly(X)$.

\begin{cor}\label{cor:ringhom}
  There is a well-defined ring homomorphism
  $\wh\chi : \QK_T^\poly(X) \to \Gamma$ defined by
  $\wh\chi(\cO^u) = \chi(q_\be) = 1$ for all $w \in W^P$ and
  $\be \in \Delta\ssm\Delta_P$.
\end{cor}
\begin{proof}
  We may consider $\Gamma$ as an algebra over $\Gamma[q]$ through the
  ring homomorphism $\mu : \Gamma[q] \to \Gamma$ defined by
  $\mu(q_\be)=1$ for $\be \in \Delta\ssm\Delta_P$ and $\mu(a)=a$ for
  $a \in \Gamma$.  We can define $\wh\chi$ as a homomorphism of
  $\Gamma[q]$-modules by setting $\wh\chi = \mu \chi$.  Since both of
  the sets $\{\cO^u \mid u \in W^P\}$ and $\{\cO_v \mid v \in W^P\}$
  are bases for $\QK_T^\poly(X)$ over $\Gamma[q]$, it follows from the
  identity
  $\wh\chi(\cO^u \star \cO_v) = 1 = \wh\chi(\cO^u) \cdot
  \wh\chi(\cO_v)$ that $\wh\chi$ is a $\Gamma[q]$-algebra
  homomorphism, as required.
\end{proof}

It would be interesting to know if \Corollary{ringhom} can be
generalized beyond the setting of flag varieties.

\bibliographystyle{amsplain}

\providecommand{\bysame}{\leavevmode\hbox to3em{\hrulefill}\thinspace}
\providecommand{\MR}{\relax\ifhmode\unskip\space\fi MR }
\providecommand{\MRhref}[2]{%
  \href{http://www.ams.org/mathscinet-getitem?mr=#1}{#2}
}
\providecommand{\href}[2]{#2}

\end{document}